\documentclass[12pt,article]{amsart}
\usepackage[top=1.15in, bottom=1.15in, left=1.17in, right=1.17in]{geometry}
\usepackage{amsthm, amsmath, amssymb, amscd, multicol, verbatim, enumerate, graphicx,xy, color}
\usepackage{tikz}
\usepackage[colorlinks=true, pdfstartview=FitV, linkcolor=blue, citecolor=blue, urlcolor=blue, breaklinks=true]{hyperref}
\usepackage[english]{babel}
\usepackage[T1]{fontenc}
\usepackage{mathrsfs}

\newcounter{cnt} 
 \makeatletter
\def\mydggeometry{\makeatletter\dg@YGRID=1\dg@XGRID=20\unitlength=0.003pt\makeatother}
\makeatother \theoremstyle{remark}
\numberwithin{equation}{section}

\theoremstyle{definition} 
\newtheorem{definition}{Definition}\theoremstyle{definition}
\newtheorem{proposition}{Proposition}
\newtheorem{theorem}{Theorem}

\newtheorem{corollary}{Corollary}
\newtheorem{remark}{Remark}

\newcommand{\ra}{\rightarrow}

\newcommand{\ff}{\mathcal{F}}
\newcommand{\pp}{\mathcal{P}}
\newcommand{\mo}{\mathcal{O}}

\newcommand{\Gr}{{\operatorname*{Gr}}}

\newcommand{\wt}{\operatorname*{wt}}

\newcommand{\Mat}{{\operatorname*{Mat}}}
\newcommand{\NO}{{\operatorname*{NO}}}

\begin{document}

\date{\today}

\title{Symmetries on plabic graphs and associated polytopes}
\author{Xin Fang}
\address{Xin Fang: University of Cologne, Mathematical Institute, Weyertal 86--90, 50931 Cologne, Germany}
\email{xinfang.math@gmail.com}
\author{Ghislain Fourier}
\address{Ghislain Fourier: Leibniz Universit\"at Hannover, Institute for Algebra, Number Theory and Discrete Mathematics, Welfengarten 1, 30167 Hannover, Germany}
\email{fourier@math.uni-hannover.de}
\maketitle

\begin{abstract}
For Grassmann varieties, we explain how the duality between the Gelfand-Tsetlin polytopes and the Feigin-Fourier-Littelmann-Vinberg polytopes arises from different positive structures.
\end{abstract}

\section{Introduction}
Plabic graphs (planar bicolored graphs) are introduced by Postnikov \cite{Pos} to parametrize cells in the totally non-negative (TNN) Grassmannians $(\mathrm{Gr}_{k,n}(\mathbb{R}))_{\geq 0}$. These graphs are drawn inside a disk with boundary vertices labelled by $1,2,\ldots,n$ in a fixed orientation and internal vertices coloured by black and white. For a reduced plabic graph $\mathcal{G}$ corresponding to the top cell in the TNN-Grassmannian $(\mathrm{Gr}_{n-k,n}(\mathbb{R}))_{\geq 0}$, Rietsch and Williams \cite{RW} constructed a family of polytopes for positive integers $r$ as Newton-Okounkov bodies \cite{KK, LM09} associated to the line bundle $r\in\mathbb{Z}\cong\textrm{Pic}(\mathrm{Gr}_{n-k,n}(\mathbb{C}))$.
\par
When the plabic graph $\mathcal{G}:=\mathcal{G}_{k,n}^{rec}$ is chosen as in \cite{RW} (see Section \ref{Sec:4.2}), the corresponding Newton-Okounkov body $\textrm{NO}_{\mathcal{G}}$ is unimodularly equivalent to the Gelfand-Tsetlin polytope $\mathrm{GT}_{n-k,n}^1$.
\par
The Newton-Okounkov body is by definition a closed convex hull of points; even when it is a polytope, to read off its defining inequalities is a hard problem. In \cite{RW}, the authors used mirror symmetry of Grassmannians to obtain these inequalities from the tropicalization of the super-potential on an open set of the mirror Grassmannian arising from the Landau-Ginzburg model. By applying this symmetry, they give explicit defining inequalities of $\textrm{NO}_{\mathcal{G}}$.
\par
Lattice points in Gelfand-Tsetlin polytopes parametrize bases of finite dimensional irreducible representations of the Lie algebra $\mathfrak{sl}_n$. Motivated by a conjecture of Vinberg, another family of polytopes, called $\mathrm{FFLV}$ polytopes, is found by Feigin, the second author and Littelmann \cite{FeFoL11} whose lattice points also parametrize bases of finite dimensional irreducible representations of $\mathfrak{sl}_n$.
\par
For a plabic graph $\mathcal{G}$, its mirror $\mathcal{G}^\vee$ is defined by swapping the black/white colouring of internal vertices in $\mathcal{G}$. When the plabic graph $\mathcal{G}$ corresponds to the top cell in $(\mathrm{Gr}_{n-k,n}(\mathbb{R}))_{\geq 0}$, $\mathcal{G}^\vee$ parametrizes the top cell in $(\mathrm{Gr}_{k,n}(\mathbb{R}))_{\geq 0}$.

\begin{theorem}
The Newton-Okounkov body $\textrm{NO}_{\mathcal{G}^\vee}$ is unimodularly equivalent to $\mathrm{FFLV}_{k,n}^1$ (see Section \ref{Sec:4.1} for definition).
\end{theorem}

Another way to relate Gelfand-Tsetlin polytopes to FFLV polytopes is via a connection between the corresponding clusters in different cluster algebras. Each reduced plabic graph $\mathcal{G}$ gives a cluster $\mathcal{C}$ consisting of Pl\"ucker coordinates $\Delta_{I_1},\ldots,\Delta_{I_m}$ where $I_1,\ldots,I_m$ are some $(n-k)$-element subsets of $[n]=\{1,2,\ldots,n\}$.
\par
For $I\subset [n]$, let $I^c$ denote its complement. Then the set $\mathcal{C}'=\{\Delta_{I_1^c},\ldots,\Delta_{I_m^c}\}$ is a cluster for $\mathrm{Gr}_{k,n}(\mathbb{C})$, corresponding to a plabic graph $\mathcal{G}^\vee$.

\begin{corollary}
The Newton-Okounkov body $\textrm{NO}_{\mathcal{G}^\vee}$ is unimodularly equivalent to $\mathrm{FFLV}_{k,n}^1$.
\end{corollary}

\subsection{Acknowledgement}
Part of this work was announced (see \cite{F16}) by X.F. in the workshop "PBW Structures in Representation Theory", held in MFO in March, 2016, he would like to thank MFO for the hospitality. The work of X.F. was partially supported by Alexander von Humboldt Foundation.

\section{Plabic graphs}

We recall the definition and basic properties of plabic graphs, following \cite{Pos, RW}.

\begin{definition}
A \emph{plabic graph} is an undirected planar graph $\mathcal{G}$ satisfying:
\begin{enumerate}
\item $\mathcal{G}$ is embedded in a closed disk and considered up to homotopy;
\item $\mathcal{G}$ has $n$ vertices on the boundary of the disk, called \emph{boundary vertices}, which are labeled clockwise by $1,2,\ldots, n$;
\item all other vertices of $\mathcal{G}$ are strictly inside the disk, they are called \emph{internal vertices} and coloured in black and white;
\item each boundary vertex is incident to a single edge.
\end{enumerate}
\end{definition}

\begin{center}
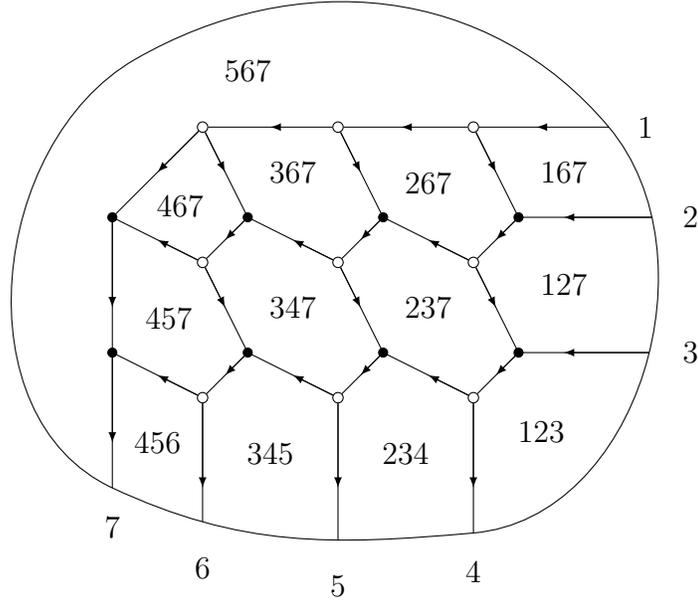
\begin{figure}
\begin{tikzpicture}[scale=.6]

    \node[right] at (6.4,3) {{$1$}};
    \node[right] at (7.4,1) {{$2$}};
    \node[right] at (7.4,-2) {{$3$}};
    \node[below] at (3,-6.4) {{$4$}};
    \node[below] at (0,-6.7) {{$5$}};
    \node[below] at (-3,-6.3) {{$6$}};
    \node[below] at (-5,-5.4) {{$7$}};

    \draw (6,3) -- (3,3) -- (0,3) -- (-3,3) -- (-5,1) -- (-5,-2) -- (-5,-5);
    \draw (6.96,1) -- (4,1) -- (3,0) -- (1,1) -- (0,0) -- (-2,1) -- (-3,0) -- (-2,-2) -- (-3,-3) -- (-3,-5.75);
    \draw (6.90,-2) -- (4,-2) -- (3,-3) -- (1,-2) -- (0,-3) -- (0,-6.16); 
    \draw (-5,1) -- (-3,0);
    \draw (-3,3) -- (-2,1);
    \draw (0,3) -- (1,1);
    \draw (3,3) -- (4,1);
    \draw (-3,-3) -- (-5,-2);
    \draw (0,0) -- (1,-2);
    \draw (0,-3) -- (-2,-2);
    \draw (3,0) -- (4,-2);
    \draw (3,-3) -- (3,-6);
    
	\draw[-latex] (4.5,3) -- (4.4,3);
	\draw[-latex] (1.5,3) -- (1.4,3);
	\draw[-latex] (-1.4,3) -- (-1.5,3);
	\draw[-latex] (-3,3) -- (-4,2);
	\draw[-latex] (-5,1) -- (-5,-1);
	\draw[-latex] (-5,-2) -- (-5,-4);
	\draw[-latex] (-3,-3) -- (-3,-5);
	\draw[-latex] (0,-3) -- (0,-5);
	\draw[-latex] (3,-3) -- (3,-5);
	\draw[-latex] (6.96,1) -- (5,1);
	\draw[-latex] (6.90,-2) -- (5,-2);
	\draw[-latex] (4,1) -- (3.5,0.5);
	\draw[-latex] (3,3) -- (3.5,2);
	\draw[-latex] (0,3) -- (0.5,2);
	\draw[-latex] (-3,3) -- (-2.5,2);
	\draw[-latex] (1,1) -- (0.5,0.5);
	\draw[-latex] (-2,1) -- (-2.5,0.5);
	\draw[-latex] (0,0) -- (-1,0.5);
	\draw[-latex] (-3,0) -- (-4,0.5);
	\draw[-latex] (3,0) -- (2,0.5);
	\draw[-latex] (-3,0) -- (-2.5,-1);
	\draw[-latex] (0,0) -- (0.5,-1);
	\draw[-latex] (3,0) -- (3.5,-1);
	\draw[-latex] (4,-2) -- (3.5,-2.5);
	\draw[-latex] (1,-2) -- (0.5,-2.5);
	\draw[-latex] (-2,-2) -- (-2.5,-2.5);
	\draw[-latex] (3,-3) -- (2,-2.5);
	\draw[-latex] (0,-3) -- (-1,-2.5);
	\draw[-latex] (-3,-3) -- (-4,-2.5);

    \draw (6,3) to [out=-50,in=5] (3,-6) to [out=185,in=-25]  (-5,-5) to [out=155,in=-150] (-4.5,4.5) to [out=30,in=130] (6,3);

    \node at (5,2) {{$167$}};
    \node at (5,-0.5) {{$127$}};
    \node at (4.5,-3.75) {{$123$}};
    \node at (2,1.75) {{$267$}};
    \node at (2,-1) {{$237$}};
    \node at (1.5,-4.25) {{$234$}};
    \node at (-1,2) {{$367$}};
    \node at (-1,-1) {{$347$}};
    \node at (-1.5,-4.25) {{$345$}};
    \node at (-3.5,1.25) {{$467$}};
    \node at (-3.75,-1.25) {{$457$}};
    \node at (-4,-4) {{$456$}};
    \node at (-2,4.25) {{$567$}};

        \draw[fill, white] (3,3) circle [radius=.115];
    \draw (3,3) circle [radius=0.115];
        \draw[fill, white] (0,3) circle [radius=.115];
    \draw (0,3) circle [radius=0.115];
        \draw[fill, white] (-3,3) circle [radius=.115];
    \draw (-3,3) circle [radius=0.115];
    
    \draw[fill] (4,1) circle [radius=0.1];
    \draw[fill] (1,1) circle [radius=0.1];
    \draw[fill] (-2,1) circle [radius=0.1];
    \draw[fill] (-5,1) circle [radius=0.1];
    
        \draw[fill, white] (3,0) circle [radius=.115];
    \draw (3,0) circle [radius=0.115];
        \draw[fill, white] (0,0) circle [radius=.115];
    \draw (0,0) circle [radius=0.115];
        \draw[fill, white] (-3,0) circle [radius=.115];
    \draw (-3,0) circle [radius=0.115];    
    
    \draw[fill] (4,-2) circle [radius=0.1];
    \draw[fill] (1,-2) circle [radius=0.1];
    \draw[fill] (-2,-2) circle [radius=0.1];
    \draw[fill] (-5,-2) circle [radius=0.1];
    
        \draw[fill, white] (3,-3) circle [radius=.115];
    \draw (3,-3) circle [radius=0.115];
        \draw[fill, white] (0,-3) circle [radius=.115];
    \draw (0,-3) circle [radius=0.115];
        \draw[fill, white] (-3,-3) circle [radius=.115];
    \draw (-3,-3) circle [radius=0.115];
   
\end{tikzpicture}
\caption{\label{Plabic}Plabic graph $\mathcal{G}$ of trip permutation $\pi_{4,7}$ and face labelling $\lambda_\mathcal{G}$}
\end{figure}
\end{center}

In \cite{Pos} (see also \cite{RW}), there are three \textit{local moves} defined on plabic graphs: gluing two vertices of the same colour; removing redundant vertices and mutating a square. For a plabic graph $\mathcal{G}$, let $\ff(\mathcal{G})$ denote the set of its faces, which is invariant under the local moves.

\begin{definition}
A plabic graph $\mathcal{G}$ is called \emph{reduced} if there are no parallel edges
    \begin{tikzpicture}[scale=0.7]
    \draw (0,0) -- (1,0);
    \draw (2,0) -- (3,0);

    \draw (1,0.06) -- (2,0.06);
    \draw (1,-.06) -- (2,-.06);

    \draw[fill] (1,0) circle [radius=0.1];
    \draw[fill, white] (2,0) circle [radius=0.1];
    \draw (2,0) circle [radius=0.1];

    \end{tikzpicture}
after applying any sequences of local moves.
\end{definition}

\begin{definition}
Let $\mathcal{G}$ be a reduced plabic graph. The \emph{trip} $T_i$ starting from a boundary vertex $i$ is the path going through the edges of $\mathcal{G}$, obeying the following rules:
\begin{enumerate}
\item at each internal black vertex, the path turns to the rightmost direction;
\item at each internal white vertex, the path turns to the leftmost direction.
\end{enumerate}
The trip $T_i$ ends at a boundary vertex ${\pi(i)}$. We associate in this way a \emph{trip permutation} $\pi_\mathcal{G}:=(\pi(1),\ldots,\pi(n))$ to $\mathcal{G}$. Let $\pi_{k,n}=(n-k+1,n-k+2,\ldots,n,1,2,\ldots,n-k)$.  The \emph{face labelling} of $\mathcal{G}$ is the injective map $\lambda_\mathcal{G}:\ff(\mathcal{G})\ra\binom{[n]}{k}$ (the set of $k$-element subsets of $\{1, \ldots, n\}$) defined as follows: for a face $F\in\ff(\mathcal{G})$, $\lambda_\mathcal{G}(F)$ consists of those $i$ such that $F$ is to the left of the trip $T_i$. We set $\mathcal{V}_\mathcal{G}:=\lambda_\mathcal{G}(\ff(\mathcal{G}))$.
\end{definition}
See Figure \ref{Plabic} for an example.

\section{Polytopes arising from plabic graphs}

We associate polytopes to plabic graphs following \cite{RW}. Let $\mathbb{K}=\mathbb{R}$ or $\mathbb{C}$ be the base field.

\subsection{Positive Grassmannians}

For $0<k<n$, let $\Mat_{k,n}$ denote the set of $k\times n$-matrices with entries in $\mathbb{K}$. For $J\in \binom{[n]}{k}$ and $A\in\Mat_{k,n}$, let $\Delta_J(A)$ denote the maximal minor of $A$ corresponding to columns in $J$. 
\par
Let $\Gr_{k,n}$ be the Grassmann variety embedded into $\mathbb{P}^{N-1}$ via the Pl\"ucker embedding where $N=\binom{n}{k}$. The minors $\{\Delta_J\mid J\in\binom{[n]}{k}\}$ give the Pl\"ucker coordinates on $\Gr_{k,n}$.
When the base field is $\mathbb{R}$, the \emph{totally non-negative (resp. totally positive) Grassmannian} $(\Gr_{k,n}(\mathbb{R}))_{\geq 0}$ consists of those elements in $\Gr_{k,n}$ having non-negative (resp. positive) Pl\"ucker coordinates.

\subsection{Perfect orientations}
To study flow models on plabic graphs, we fix a perfect orientation $\mathcal{O}$ on $\mathcal{G}$. Such an orientation requires at each black (resp. white) internal vertex there is exactly one edge going out (resp. going in). It is shown in \cite{PSW09} that each reduced plabic graph admits an acyclic perfect orientation. Once such an orientation is fixed, we denote the source set by $I_\mathcal{O}:=\{i\in[n]\mid i\ \text{is a boundary source of }\mathcal{O}\}$; its complement $I^c_{\mathcal{O}}$ is the set of boundary sinks.
\par
For $I\in\binom{[n]}{k}$, let $x_I$ be a variable. For $i\in I_\mo$ and $j\in I^c_\mo$, let $\pp_{i,j}$ be the set of directed paths from $i$ to $j$. For such a directed path $\gamma$, let $\ff_\gamma(\mathcal{G})$ denote the set of faces to the left of $\gamma$.
A flow $\mathfrak{F}$ from $I_\mo$ to $J\in\binom{[n]}{k}$ is a collection of pairwise vertex-disjoint directed paths in $\mathcal{G}$ going from  $I_\mo\backslash (I_\mo\cap J)$ to $J\backslash (I_\mo\cap J)$. 
\par
For a directed path $\gamma\in\pp_{i,j}$, we define the \emph{weight} of $\gamma$ in $\mathbb{C}[x_I\mid I\in\binom{[n]}{k}]$ by:
$$\wt(\gamma):=\prod_{F\in\ff_\gamma(\mathcal{G})}x_{\lambda_\mathcal{G}(F)}.$$ 
The weight of a flow is the product of the weights of the paths it contains. For $J\in\binom{[n]}{k}$, we define $P_J$ to be the sum of the weights of all flows from $I_\mo$ to $J$.

For a reduced plabic graph $\mathcal{G}$ of trip permutation $\pi_{n-k,k}$ with perfect orientation $\mo$, there exists only one face $F_\emptyset$ to the right of all directed paths with $\lambda_\mathcal{G}(F_\emptyset)=\{n-k+1,\cdots,n\}$. We set $\mathcal{V}_\mathcal{G}^\circ:=\mathcal{V}_\mathcal{G}\backslash\{\lambda_\mathcal{G}(F_\emptyset)\}$, $\Delta_\mathcal{G}:=\{x_I\mid I\in\mathcal{V}_\mathcal{G}\}$ and $\Delta_\mathcal{G}^\circ:=\{x_I\mid I\in\mathcal{V}^\circ_\mathcal{G}\}$.

\begin{theorem}[\cite{Pos, Tal08}]\label{Thm:PosTal}
Let $\mathbb{X}:=\Gr_{k,n}(\mathbb{C})$ and $\mathbb{C}(\mathbb{X})$ be the field of rational functions on $\mathbb{X}$. There exists an isomorphism of fields:
$$\mathbb{C}(\mathbb{X})\cong \mathbb{C}(x_I\mid x_I\in \Delta_\mathcal{G}^\circ),\ \ \Delta_J\mapsto P_J.$$
\end{theorem}

The choice of the perfect orientation $\mo$ will only change the formula of $P_J$ by a scalar. We always assume that the choice $I_\mo=\{1,2,\cdots,k\}$ is made.
\par
Let $<$ be a total order on $\Delta_\mathcal{G}$. It induces a term order $<$ on monomials in $\Delta_\mathcal{G}$ by taking the lexicographic order. Let $f$ be a polynomial in Pl\"ucker coordinates of $\mathbb{X}$. By Theorem \ref{Thm:PosTal}, $f$ can be written as a polynomial in $\Delta_\mathcal{G}^\circ$:
$$f=\sum_{\bold{a}\in\mathbb{Z}^{\mathcal{V}_\mathcal{G}^\circ}}c_\bold{a}x^\bold{a},\ \ \text{where}\ \ x^\bold{a}=\prod_{I\in\mathcal{V}_\mathcal{G}^\circ}x_I^{a_I}\text{  if  }\mathbf{a}=(a_I)_{I\in\mathcal{V}_\mathcal{G}^\circ}.$$

Let $\nu_\mathcal{G}:\mathbb{C}(\mathbb{X})^*\ra \mathbb{Z}^{\mathcal{V}_\mathcal{G}^\circ}$ be the minimal term valuation on $\mathbb{C}(\mathbb{X})$ with respect to the above total order.

Let $\mathcal{L}_k$ denote the very ample line bundle on $\mathbb{X}$ generating $\mathrm{Pic}(\mathbb{X})$. It gives the Pl\"ucker embedding. The space of global sections $\mathrm{H}^0(\mathbb{X},\mathcal{L}_k^r)$, as a representation of $\mathrm{GL}_n(\mathbb{C})$, is isomorphic to $V(r\varpi_k)^*$, where the latter is the dual of the finite dimensional irreducible representation of highest weight $r\varpi_k$ ($\varpi_k$ is the $k$-th fundamental weight). The homogeneous coordinate ring $\mathbb{C}[\mathbb{X}]:=\bigoplus_{r\geq 0}\mathrm{H}^0(\mathbb{X},\mathcal{L}_k^r)$ is embedded into $\mathbb{C}(\mathbb{X})$ by sending $s\in\mathrm{H}^0(\mathbb{X},\mathcal{L}_k^r)$ to $s/\Delta_{[k]}^r$.

\begin{definition}
The Newton-Okounkov body associated to $\mathcal{L}_k$, $\nu_\mathcal{G}$ and the lexicographic order is defined by:
$$\mathrm{NO}_\mathcal{G}:=\overline{\mathrm{conv}\left(\bigcup_{r\geq 1}\left\{\nu_\mathcal{G}(s)/r\mid s\in\mathrm{H}^0(\mathbb{X},\mathcal{L}_k^r)\backslash\{0\}\right\}\right)}.$$
\end{definition}
We set $\mathrm{NO}_\mathcal{G}^1:=\mathrm{conv}(\{\nu_\mathcal{G}(s)\mid s\in\mathrm{H}^0(\mathbb{X},\mathcal{L}_k)\backslash\{0\}\})\subseteq \mathrm{NO}_\mathcal{G}$. For the issue on whether this inclusion is proper (i.e, whether $\mathrm{NO}_\mathcal{G}$ is integral), see \cite[Theorem 15.17]{RW}.

\section{Duality between Newton-Okounkov bodies}

\subsection{Order polytopes and chain polytopes}\label{Sec:4.1}

Let $(P,\leq_P)$ be a poset with covering relation $\prec$. Stanley \cite{St} associated two Ehrhart equivalent polytopes, the order polytope and the chain polytope, to this poset. We recall here a dilated version of them. 
\par
For $r\in\mathbb{N}_{>0}$, we denote the dilated order polytope $\mo(P,r)$ to be the representation of the poset $P$ on the interval $[0,r]$ with the order on real numbers:
$$\mo(P,r):=\mathrm{Hom}_{\mathrm{Poset}}((P,\leq_P),([0,r],\leq))\subseteq\mathbb{R}^P.$$
The dilated chain polytope $\mathcal{C}(P,r)\subseteq\mathbb{R}^P$ has the following facets: for any $p\in P$, $x_p\geq 0$; for any maximal chain $p_1\prec\cdots\prec p_s$, $x_{p_1}+\cdots+x_{p_s}\leq r$, where $x_p$ is the coordinate of $p\in P$ in $\mathbb{R}^P$.
\par
Stanley \cite{St} showed that the integral points of the chain polytope $\mathcal{C}(P,1)$ are given by the characteristic functions of the anti-chains in $P$. In particular, the element $p\in P$ gives an integral point $\chi_p$ in $\mathcal{C}(P,1)$.
\par
In the following, we fix $1 \leq k \leq n-1$, and let $(P_{k,n}, \leq)$ be the poset given by the elements $p_{i,j}$, where $1 \leq i \leq k$ and $k+1 \leq j \leq n$, with covering relations
\[
 p_{i+1,j}\prec p_{i,j}  \text{ and }\  p_{i, j+1} \prec p_{i,j}.
\]
The polytope $\mathcal{O}(P_{k,n},r)$ is the Gelfand-Tsetlin polytope $\mathrm{GT}_{k,n}^r$ for the representation $V(r\varpi_k)$ of $\mathfrak{sl}_n$ (\cite{GT50});
while $\mathcal{C}(P_{k,n},r)$ is the Feigin-Fourier-Littelmann-Vinberg polytope $\mathrm{FFLV}_{k,n}^r$ (\cite{ABS11, FeFoL11}) of the same representation.

For a polytope $Q\subset\mathbb{R}^m$, let $S(Q):=Q\cap\mathbb{Z}^m$ denote the set of integral points in it. The following integer decomposition properties hold: the $r$-fold Minkowski sum of $S(\mo(P_{k,n},1))$ (resp. $S(\mathcal{C}(P_{k,n},1)$)) coincides with $S(\mo(P_{k,n},r))$ (resp. $S(\mathcal{C}(P_{k,n},r))$).

Moreover, if $\mathbf{a} = \{p_{i_1, j_1}, \ldots, p_{i_s, j_s}\}$ is an anti-chain in $P_{k,n}$, then one has for the corresponding lattice points $\chi_\mathbf{a} = \chi_{p_{i_1, j_1}} + \ldots + \chi_{p_{i_s, j_s}}\in\mathcal{C}(P_{k,n},1)$.

\begin{proposition}\label{prop-additive}
Suppose $Q$ is an integral polytope in $\mathbb{R}^{P_{k,n}}$ such that 
\begin{itemize}
\item $\# S(Q)= \# S(\mathrm{FFLV}_{k,n}^1)$;
\item there is a parametrization of the lattice points in $Q$ by anti-chains in $P_{k,n}$ sending an anti-chain $\mathbf{a}$ to $y_\mathbf{a}\in\mathbb{R}^{P_{k,n}}$ such that for any anti-chain $\mathbf{a} = \{p_{i_1, j_1}, \ldots, p_{i_s, j_s}\}$ the relation $y_\mathbf{a} = y_{p_{i_1, j_1}} + \ldots + y_{p_{i_s, j_s}}$ holds;
\item there is a linear map of determinant $1$ expressing $y_{p_{i,j}}$ in terms of $\chi_{p_{i,j}}$.
\end{itemize}
Then the assignment $\chi_{p_{i,j}} \mapsto y_{p_{i,j}}$ induces a unimodularly equivalence from $\mathrm{FFLV}_{k,n}^1$ to $Q$.
\end{proposition}

\subsection{Duality of polytopes from positive structures}\label{Sec:4.2}

We refer to \cite[Section 7.1]{RW} for the definition of the rec-plabic graph $\mathcal{G}_{k,n}^{rec}$. For example, the plabic graph in Figure \ref{Plabic} is $\mathcal{G}_{4,7}^{rec}$.

The following has been shown in {\cite[Lemma 15.2]{RW}}: 
\begin{proposition}
The Newton-Okounkov body $\NO_{\mathcal{G}_{k,n}^{rec}}$ is unimodularly equivalent to the Gelfand-Tsetlin polytope $\mathrm{GT}_{n-k,n}^1$.
\end{proposition}

We define the dual rec-plabic graph $(\mathcal{G}_{k,n}^{rec})^\vee$ by swapping the black/white colour of the internal vertices, reversing the perfect orientation and changing the boundary labelling $r\mapsto r+n-k\text{ mod }n$. The dual rec-plabic graph is a plabic graph of trip permutation $\pi_{k,n}$ with a perfect orientation. The face labelling in $(\mathcal{G}_{k,n}^{rec})^\vee$ of a face $F$ in $\mathcal{G}_{k,n}^{rec}$ is given by the complement:
$$\lambda_{(\mathcal{G}_{k,n}^{rec})^\vee}(F)=(\lambda_{\mathcal{G}_{k,n}^{rec}}(F))^c.$$

Notice that in $(\mathcal{G}_{k,n}^{rec})^\vee$, for a boundary source $i$ and a boundary sink $j$, the flow from $i$ to $j$ of strongly minimal weight (we borrow the notion of strongly minimal from \cite[Definition 5.13]{RW}) is given by a "vertical" path starting from $i$ followed by a "horizontal" path ending in $j$. We denote this path by $\gamma^{\min}_{i,j}$ (see Figure \ref{Plabicrev} for an example for $\gamma^{\min}_{3,6}$).

\begin{center}
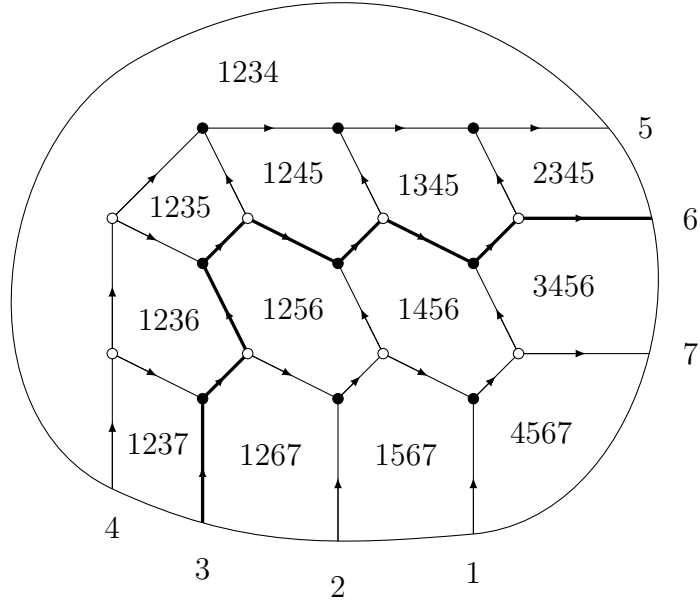
\begin{figure}
\begin{tikzpicture}[scale=.6]

    \node[right] at (6.4,3) {{$5$}};
    \node[right] at (7.4,1) {{$6$}};
    \node[right] at (7.4,-2) {{$7$}};
    \node[below] at (3,-6.4) {{$1$}};
    \node[below] at (0,-6.7) {{$2$}};
    \node[below] at (-3,-6.3) {{$3$}};
    \node[below] at (-5,-5.4) {{$4$}};

    \draw (6,3) -- (3,3) -- (0,3) -- (-3,3) -- (-5,1) -- (-5,-2) -- (-5,-5);
    \draw [line width=1.25pt] (6.96,1) -- (4,1);
    \draw [line width=1.25pt] (4,1) -- (3,0);
    \draw [line width=1.25pt] (3,0) -- (1,1);
    \draw [line width=1.25pt] (1,1) -- (0,0);
    \draw [line width=1.25pt] (0,0) -- (-2,1); 
    \draw [line width=1.25pt] (-2,1) -- (-3,0);
    \draw [line width=1.25pt] (-3,0) -- (-2,-2);
    \draw [line width=1.25pt] (-3,-3) -- (-2,-2);    
    \draw [line width=1.25pt] (-3,-3) -- (-3,-5.75);
    \draw (6.90,-2) -- (4,-2) -- (3,-3) -- (1,-2) -- (0,-3) -- (0,-6.16); 
    \draw (-5,1) -- (-3,0);
    \draw (-3,3) -- (-2,1);
    \draw (0,3) -- (1,1);
    \draw (3,3) -- (4,1);
    \draw (-3,-3) -- (-5,-2);
    \draw (0,0) -- (1,-2);
    \draw (0,-3) -- (-2,-2);
    \draw (3,0) -- (4,-2);
    \draw (3,-3) -- (3,-6);
    
	\draw[-latex]  (4.4,3) -- (4.5,3);
	\draw[-latex]  (1.4,3) -- (1.5,3);
	\draw[-latex]  (-1.5,3) -- (-1.4,3);
	\draw[-latex] (-5,1) -- (-4,2);
	\draw[-latex] (-5,-2) -- (-5,-0.5);
	\draw[-latex] (-5,-5) -- (-5,-3.5);
	
	\draw[-latex] (-3,-5.75) -- (-3,-4.5);
	\draw[-latex] (0,-6.16) -- (0,-4.75);
	\draw[-latex] (3,-6) -- (3,-4.75);
	\draw[-latex] (4,1) -- (5.5,1);
	\draw[-latex] (4,-2) -- (5.5,-2);
	\draw[-latex] (3,0) -- (3.5,0.5);
	\draw[-latex] (4,1) -- (3.5,2);
	\draw[-latex] (1,1) -- (0.5,2);
	\draw[-latex] (-2,1) -- (-2.5,2);
	\draw[-latex] (0,0) -- (0.5,0.5);
	\draw[-latex] (-3,0) -- (-2.5,0.5);
	\draw[-latex] (-2,1) -- (-1,0.5);
	\draw[-latex] (-5,1) -- (-4,0.5);
	\draw[-latex] (1,1) -- (2,0.5);
	\draw[-latex] (-2,-2) -- (-2.5,-1);
	\draw[-latex] (1,-2) -- (0.5,-1);
	\draw[-latex] (4,-2) -- (3.5,-1);
	\draw[-latex] (3,-3) -- (3.5,-2.5);
	\draw[-latex] (0,-3) -- (0.5,-2.5);
	\draw[-latex] (-3,-3) -- (-2.5,-2.5);
	\draw[-latex] (1,-2) -- (2,-2.5);
	\draw[-latex] (-2,-2) -- (-1,-2.5);
	\draw[-latex] (-5,-2) -- (-4,-2.5);

    \draw (6,3) to [out=-50,in=5] (3,-6) to [out=185,in=-25]  (-5,-5) to [out=155,in=-150] (-4.5,4.5) to [out=30,in=130] (6,3);

    \node at (5,2) {{$2345$}};
    \node at (5,-0.5) {{$3456$}};
    \node at (4.5,-3.75) {{$4567$}};
    \node at (2,1.75) {{$1345$}};
    \node at (2,-1) {{$1456$}};
    \node at (1.5,-4.25) {{$1567$}};
    \node at (-1,2) {{$1245$}};
    \node at (-1,-1) {{$1256$}};
    \node at (-1.5,-4.25) {{$1267$}};
    \node at (-3.5,1.25) {{$1235$}};
    \node at (-3.75,-1.25) {{$1236$}};
    \node at (-4,-4) {{$1237$}};
    \node at (-2,4.25) {{$1234$}};

        \draw[fill] (3,3) circle [radius=.115];
        \draw[fill] (0,3) circle [radius=.115];
        \draw[fill] (-3,3) circle [radius=.115];
        
    \draw[fill, white] (4,1) circle [radius=0.1];
        \draw (4,1) circle [radius=0.115];
    \draw[fill, white] (1,1) circle [radius=0.1];
        \draw (1,1) circle [radius=0.115];
    \draw[fill, white] (-2,1) circle [radius=0.1];
	    \draw (-2,1) circle [radius=0.115];
    \draw[fill, white] (-5,1) circle [radius=0.1];
  	  \draw (-5,1) circle [radius=0.115];
    
        \draw[fill] (3,0) circle [radius=.115];
        \draw[fill] (0,0) circle [radius=.115];
        \draw[fill] (-3,0) circle [radius=.115];
    
    \draw[fill, white] (4,-2) circle [radius=0.1];
  	  \draw (4,-2) circle [radius=0.115];
    \draw[fill, white] (1,-2) circle [radius=0.1];
   		 \draw (1,-2) circle [radius=0.115];
    \draw[fill, white] (-2,-2) circle [radius=0.1];
    		\draw (-2,-2) circle [radius=0.115];
    \draw[fill, white] (-5,-2) circle [radius=0.1];
	    \draw (-5,-2) circle [radius=0.115];

        \draw[fill] (3,-3) circle [radius=.115];
        \draw[fill] (0,-3) circle [radius=.115];
        \draw[fill] (-3,-3) circle [radius=.115];
   
\end{tikzpicture}
\caption{\label{Plabicrev}Plabic graph $\mathcal{G}^\vee$, with a minimal path from $3$ to $6$}
\end{figure}
\end{center}

\begin{proposition} \label{prop-add2}
In the dual rec-plabic graph $(\mathcal{G}_{k,n}^{rec})^\vee$, let $\{i_1 < \ldots < i_r\}$ be a subset of the sources and $\{j_1 > \ldots > j_r\}$ be a subset of the sinks. Let $J = \{i_1, \ldots, i_r, j_1, \ldots, j_r\}$. Then the unique flow $\mathcal{F}(J)$ of strongly minimal 
weight is given by $\{ \gamma^{\min}_{i_1,j_1}, \ldots, \gamma^{\min}_{i_r,j_r} \}$.
\end{proposition}
\begin{proof}
Since the paths of strongly minimal weight do not intersect, the flow of minimal weight is given by the union of these paths. 
\end{proof}

\begin{theorem}\label{rec-fflv}
The Newton-Okounkov body $\NO_{(\mathcal{G}_{k,n}^{rec})^\vee}$ is unimodularly equivalent to the FFLV polytope $\mathrm{FFLV}_{k,n}^1$.
\end{theorem}
\begin{proof}
We first set $Q=\mathrm{NO}_{(\mathcal{G}_{k,n}^{rec})^\vee}^1$ and verify the conditions in Proposition~\ref{prop-additive} to show that $Q$ is unimodularly equivalent to $\mathrm{FFLV}_{k,n}^1$ by a linear map.
\par
The polytope $Q$ is a lattice polytope satisfying $\#S(Q)=\#S(\mathrm{FFLV}_{k,n}^1)$ (the valuation images of the Pl\"ucker coordinates are different). Let $f_{i\times j} := \nu_{(\mathcal{G}_{k,n}^{rec})^\vee} (\gamma_{i,j}^{\min})$. We define a linear map 
\[\psi: \mathrm{FFLV}_{k,n}^1 \longrightarrow Q, \; \chi_{p_{i,j}} \mapsto f_{i\times j}.\]
We label a basis on the right hand side indexed by the faces of the plabic graph and a basis on the left hand side indexed by the elements $p_{i,j}$. Using row operations, one can show straightforwardly, that the matrix of $\psi$ corresponding to these bases has determinant $1$.
\par
Since $\psi$ is linear, $\mathrm{NO}_{(\mathcal{G}_{k,n}^{rec})^\vee}$ is unimodularly equivalent to $\mathrm{FFLV}_{k,n}^1$.
\end{proof}

\begin{remark}
We set $(\mathcal{G}_{k,n}^{rec})_{w_0}$ to be the plabic graph obtained from $\mathcal{G}_{k,n}^{rec}$ by replacing each $ I = \{i_1, \ldots, i_{n-k}\}$ by $I_{w_0} = \{ n +1 - i_{n-k}, \ldots, n+1 - i_1\}$. This is nothing but applying a maximal Green sequence of mutations \cite{Kel} to the cluster variables in $\mathcal{G}_{k,n}^{rec}$. Then one can show similarly to the theorem above, that the Newton-Okounkov body $\NO_{(\mathcal{G}_{k,n}^{rec})_{w_0}}$ is unimodularly equivalent to $\mathrm{FFLV}_{n-k,n}^1$.
\end{remark}

\end{document}